\documentclass[a4paper]{article}

\usepackage{amssymb,amsmath,amsthm,amsopn}
\usepackage{mathtools}
\usepackage{graphicx}

\DeclareGraphicsExtensions{.pdf,.png,.eps}
\graphicspath{{figs/}}

\usepackage[hidelinks]{hyperref}
\usepackage[capitalize]{cleveref}
\usepackage{enumitem}
\usepackage{pifont}
\usepackage[dvipsnames]{xcolor}

\newcommand{\Host}{\ensuremath{H}}
\newcommand{\Guest}{\ensuremath{G}}

\newcommand{\Guestclass}{\ensuremath{\mathfrak{F}}}

\newcommand{\CBipartites}{{\ensuremath{\mathfrak{C}\mathfrak{B}}}}
\newcommand{\Differences}{\ensuremath{\mathfrak{D}}}

\newcommand{\cn}[3]{\ensuremath{\operatorname{c}_{\ensuremath{#1}}^{#2}(#3)}}

\makeatletter
\def\moverlay{\mathpalette\mov@rlay}
\def\mov@rlay#1#2{\leavevmode\vtop{%
   \baselineskip\z@skip \lineskiplimit-\maxdimen
   \ialign{\hfil$\m@th#1##$\hfil\cr#2\crcr}}}
\newcommand{\charfusion}[3][\mathord]{
    #1{\ifx#1\mathop\vphantom{#2}\fi
        \mathpalette\mov@rlay{#2\cr#3}
      }
    \ifx#1\mathop\expandafter\displaylimits\fi}
\makeatother
\newcommand{\cupdot}{\charfusion[\mathbin]{\cup}{\cdot}}

\newcommand{\disver}{\ensuremath{\mathop{\cupdot}}}

\DeclareMathOperator{\ldim}{ldim} 

\newtheorem{theorem}{Theorem}
\crefname{theorem}{Theorem}{Theorems}
\newtheorem{lemma}[theorem]{Lemma}
\crefname{lemma}{Lemma}{Lemmas}
\newtheorem{corollary}[theorem]{Corollary}
\crefname{corollary}{Corollary}{Corollaries}

\crefname{figure}{Figure}{Figures}

\def\ovl{\overline}

\begin{document}

\title{A Note on Covering Young Diagrams with Applications to Local Dimension of Posets}

\author{
 Stefan Felsner, 
 Torsten Ueckerdt 
}

\maketitle

\begin{abstract}
  We prove that in every cover of a Young diagram with $\binom{2k}{k}$
  steps with generalized rectangles there is a row or a column in the
  diagram that is used by at least $k+1$ rectangles.  We show that
  this is best-possible by partitioning any Young diagram with
  $\binom{2k}{k}-1$ steps into actual rectangles, each row and each
  column used by at most $k$ rectangles.  This answers two questions
  by Kim~\textit{et al.}~\cite{KMMSSUW18}.
 
  Our results can be rephrased in terms of local covering numbers of
  difference graphs with complete bipartite graphs, which has
  applications in the recent notion of local dimension of partially
  ordered sets.
\end{abstract}

\section{Introduction}

Let $\mathbb{N}$ denote the set of all natural numbers (i.e., positive
integers).  For $x \in \mathbb{N}$ we denote $[x] = \{1,\ldots,x\}$ to
be the set of the first $x$ natural numbers.  A \emph{Young diagram}
with $r$ rows and $c$ columns is a subset $Y \subseteq [r] \times [c]$
such that whenever $(i,j) \in Y$, then $(i-1,j) \in Y$ provided $i
\geq 2$, as well as $(i,j-1) \in Y$ provided $j \geq 2$.  A Young
diagram\footnote{In the literature our Young diagrams are more
  frequently called Ferrers diagrams. We stick to Young diagram to be
  consistent with \cite{KMMSSUW18}.} 
is visualized as a set of axis-aligned unit squares that are
arranged consecutively in rows and columns, each row starting in the
first column, and with every row (except the first) being at most as
long as the row above.  The number of steps of a Young diagram $Y$ is
the number of different row lengths in $Y$, i.e., the cardinality of
\[ Z = \{ (s,t) \in Y \mid (s+1,t) \notin Y \text{ and } (s,t+1)
\notin Y\},
\] where elements of $Z$ are called \emph{steps} of $Y$.  Young
diagrams with $n$ elements, $r$ rows, $c$ columns, and $z$ steps,
visualize partitions of the natural number $n$ into~$r$ unlabeled
positive integer summands (summand $s$ being the length of row $s$)
with summands on $z$ different values and largest summand being $c$.

\begin{figure}[t]
 \centering
 \includegraphics{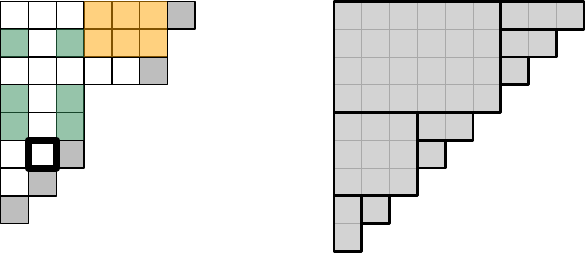}
 \caption{
 \textbf{Left:} A Young diagram $Y$ with $r=8$ rows, $c=7$ columns,
and $z=5$ steps.  Highlighted are the set $Z$ of steps (gray), the
element $(i,j) = (6,2) \in Y$ (bold boundary), the generalized
rectangle $\{2,4,5\} \times \{1,3\}$ (green), and the actual rectangle
$\{1,2\} \times \{4,5,6\}$ (orange).  \textbf{Right:} The Young diagram
$Y_9$ with $9$ steps and a $(2,3)$-local partition of $Y$ with actual
rectangles.}
 \label{fig:Young-diagram}
\end{figure}

A \emph{generalized rectangle} in a Young diagram
$Y \subseteq [r] \times [c]$ is a set $R$ of the form $R = S \times T$
with $S \subseteq [r]$ and $T \subseteq [c]$ and $R\subseteq Y$. Note
that (unless $Y = [r]\times [c]$) not every set of the form
$R = S \times T$ with $S \subseteq [r]$ and $T \subseteq [c]$
satisfies $R \subseteq Y$.  A generalized rectangle $R = S \times T$
with $S$ being a set of consecutive numbers in $[r]$ and $T$ being a
set of consecutive numbers in $[c]$ is an \emph{actual rectangle}.  A
generalized rectangle $R = S \times T$ \emph{uses} the rows in $S$ and
the columns in $T$.  See the left of \cref{fig:Young-diagram} for an
illustrative example.

Motivated by applications for the local dimension of partially ordered
sets, we investigate covering a Young diagram $Y$ with generalized
rectangles such that every row and every column of $Y$ is used by as
few generalized rectangles in the cover as possible.  We say that $Y$
is \emph{covered} by a set $C$ of generalized rectangles if $Y =
\bigcup_{R \in C} R$, i.e., $Y$ is the union of all rectangles in $C$.
In this case we also say that $C$ is a \emph{cover} of $Y$.  If
additionally the rectangles in $C$ are pairwise disjoint, we call $C$
a \emph{partition} of $Y$.  For example, the right of
\cref{fig:Young-diagram} shows a Young diagram with a partition into
actual rectangles.

\begin{theorem}\label{thm:main-simple} For any $k \in \mathbb{N}$, any
Young diagram $Y$ can be covered by a set $C$ of generalized
rectangles such that each row and each column of $Y$ used by at most
$k$ rectangles in $C$ if and only if $Y$ has strictly less than
$\binom{2k}{k}$ steps.
\end{theorem}

We prove \cref{thm:main-simple} in \cref{sec:proof-main}, answer the
questions raised by Kim~\textit{et al.} in \cref{sec:local-covering},
and describe the application to local dimension of posets in
\cref{sec:local-dimension}.

\section{Proof of \cref{thm:main-simple}}
\label{sec:proof-main}

Throughout we shall simply use the term \emph{rectangle} for
generalized rectangles, and rely on the term \emph{actual rectangle}
when specifically meaning rectangles that are contiguous.  For a Young
diagram $Y$ and $i,j \in \mathbb{N}$, let us define a cover $C$ of $Y$
to be \emph{$(i,j)$-local} if each row of $Y$ is used by at most $i$
rectangles in $C$ and each column of $Y$ is used by at most $j$
rectangles in $C$. 
For $z \in \mathbb{N}$, let $Y_z = \{ (s,t) \in [z] \times [z] \mid s
+ t \leq z+1 \}$ be the (unique) Young diagram with $z$ rows, $z$
columns, and $z$ steps.
See the right of \cref{fig:Young-diagram}.

We start with a lemma stating that instead of considering any Young diagram with $z$ steps, we may restrict our attention to just $Y_z$.

\begin{lemma}\label{lem:reduction-to-Yk}
  Let $i,j,z \in \mathbb{N}$ and $Y$ be any Young diagram with $z$
  steps.  Then~$Y$ admits an $(i,j)$-local cover if and only if $Y_z$
  admits an $(i,j)$-local cover with exactly $z$ rectangles.
\end{lemma}
\begin{proof}
First assume that $Y$ admits an $(i,j)$-local cover $C$.
If $C$ consists of strictly more than $z$ rectangles, then there are
$R_1,R_2 \in C$, $R_1 \neq R_2$, such that $R_1,R_2 \subseteq [s]
\times [t]$ for some step $(s,t) \in Z$.  However, in this case $C -
\{R_1,R_2\} + \{R_1 \cup R_2\}$ is also an $(i,j)$-local cover of $Y$
with one rectangle less.  Thus, by repeating this argument, we may
assume that $|C| = z$.
 
 \begin{figure}[t] \centering \includegraphics{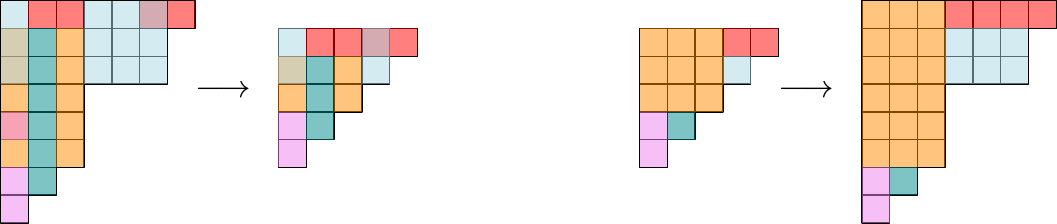}
  \caption{ Transforming a cover of any Young diagram $Y$ with $5$
steps into a cover of $Y_5$ (left) and vice versa (right).  }
  \label{fig:Yk-lemma}
 \end{figure}
 
If $Y \neq Y_z$, there is a row $s$ or a column $t$ that is not used
by any step in $Z$.  
Apply the mapping $\mathbb{N} \times \mathbb{N} \to
\mathbb{N} \times \mathbb{N}$ with
\begin{eqnarray*}
(x,y) \mapsto
  \begin{cases} (x,y) & \text{ if } x < s\\
              (x-1,y) & \text{ if } x \geq s
  \end{cases}
& \text{ respectively } &
(x,y) \mapsto
  \begin{cases} (x,y) & \text{ if } y < t\\
              (x,y-1) & \text{ if } y \geq t
  \end{cases}
\end{eqnarray*}
Intuitively, we cut out row~$s$ (respectively column $t$), moving all
rows below one step up (respectively all columns to the right one step
left).  This gives an $(i,j)$-local cover of a smaller Young diagram
with $z$ steps, and eventually leads to an $(i,j)$-local cover of
$Y_z$, as desired.  See the left of \cref{fig:Yk-lemma}.
\medskip
 
 On the other hand, if $Y_z$ admits an $(i,j)$-local cover $C =
\{R_1,\ldots,R_z\}$, this defines an $(i,j)$-local cover of $Y$ as
follows.  Index the rows used by the steps $Z$ of $Y$ by $s_1 < \cdots < s_z$ and the
columns used by the steps $Z$ of $Y$ by $t_1 < \cdots < t_z$ and let $s_0 = t_0 = 0$.
Defining
 \[ R'_a = \{(s,t) \in Y \mid s_{x-1} < s \leq s_x \text{ and }
t_{y-1} < t \leq t_y \text{ for some } (x,y) \in R_a\} 
 \] for $a = 1,\ldots,z$ gives an $(i,j)$-local cover
$\{R'_1,\ldots,R'_z\}$ of $Y$.  See the right of \cref{fig:Yk-lemma}.

Observe that the construction maps an actual rectangle $R_a$ of $Y_z$
to an actual rectangle $R'_a$ of $Y$.  
Also, if $\{R_1,\ldots,R_z\}$ is a partition of $Y_z$, then $\{R'_1,\ldots,R'_z\}$ is a partition of $Y$.
This will be used in the proof
of \cref{enum:construction} of \cref{thm:main-general}.
\end{proof}

Let us now turn to our main result.
In fact, we shall prove the following strengthening of \cref{thm:main-simple}.

\begin{theorem}\label{thm:main-general}
For any $i,j,z \in \mathbb{N}$
and any Young diagram $Y$ with $z$ steps, the following hold.
 \begin{enumerate}[label = (\roman*)]
  \item If $z < \binom{i+j}{i}$, then there exists an $(i,j)$-local
    partition of $Y$ with actual rectangles.
    \label{enum:construction}
  \item If $z \geq \binom{i+j}{i}$, then there exists no $(i,j)$-local
    cover of $Y$ with generalized rectangles.
    \label{enum:argument}
 \end{enumerate}
\end{theorem}
\begin{proof}
 First, let us prove \cref{enum:construction}.
 For shorthand notation, we define $f(i,j) := \binom{i+j}{i} - 1$.
 It will be crucial for us that the numbers $\{f(i,j)\}_{i,j \geq 1}$ solve the recursion
\begin{align}
 f(i,j) = \begin{cases}
           f(i-1,j)+f(i,j-1)+1 & \text{ if } i,j \geq 2\\
           j & \text{ if } i=1, j \geq 1\\
           i & \text{ if } i \geq 1, j=1.
          \end{cases}\label{eq:recursion}
\end{align}
This follows directly from Pascal's rule
$\binom{a}{b} = \binom{a-1}{b-1} + \binom{a-1}{b}$ for any
$a,b \in \mathbb{N}$ with $1 \leq b \leq a-1$.

Due to \cref{lem:reduction-to-Yk} it suffices to show that for any
$i,j \in \mathbb{N}$ and $z = f(i,j) = \binom{i+j}{i}-1$, there is an
$(i,j)$-local partition of $Y_z$ with actual rectangles.

We define the $(i,j)$-local partition $C$ by induction on $i$ and $j$.
For illustrations refer to~\cref{fig:construction}.

\begin{figure}
 \centering
 \includegraphics{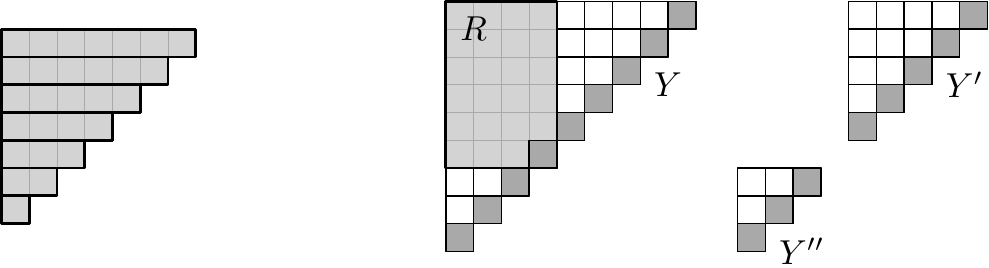}
 \caption{
  \textbf{Left:} The Young diagram $Y_z$ with $z = f(1,7) = \binom{1+7}{1} - 1 = 7$ steps and a $(1,7)$-local partition of $Y_z$ into actual rectangles.
  \textbf{Right:} The Young diagram $Y_z$ with $z = f(3,2) = \binom{3+2}{3} - 1 = 9$ steps, the rectangle $R = [a] \times [z+1-a] = [6] \times [4]$ with $a = f(2,2)+1 = 6$, and the Young diagrams $Y'$ and $Y''$ with $f(2,2) = 5$ and $f(3,1) = 3$ steps, respectively.
 }
 \label{fig:construction}
\end{figure}

If $i=1$, respectively $j=1$, then $C$ is the set of rows of $Y_j$,
respectively the set of columns of $Y_i$.
If $i \geq 2$ and $j \geq 2$, then $z = f(i,j) = f(i-1,j) + f(i,j-1) + 1$ by~\eqref{eq:recursion}.
Consider the actual rectangle $R = [a] \times [z+1-a]$ for $a = f(i-1,j) + 1$.
Then $Y_z - R$ splits into a right-shifted copy $Y'$ of $Y_{a-1}$ and
a down-shifted copy~$Y''$ of $Y_{z-a}$.
Note that $a-1 = f(i-1,j)$ and $z-a = f(i,j-1)$.

By induction we have an $(i-1,j)$-local cover $C'$ of $Y'$
and an $(i,j-1)$-local cover $C''$ of $Y''$, each consisting
of pairwise disjoint actual rectangles.
Define
  \[
   C = \{R\} \cup C' \cup C'',
 \]
this is a cover of $Y_z$ consisting of pairwise disjoint actual
rectangles.  Rows $1$ to~$a$ are used by $R$ and at most $i-1$
rectangles in $C'$, and rows $a+1$ to $z$ are used by at most~$i$
rectangles in $C''$.  Hence each row of $Y_z$ is used by at most $i$
rectangles in $C$.  Similarly each column of $Y_z$ is used by at most
$j$ rectangles in~$C$.  Thus $C$ is an $(i,j)$-local partition of
$Y_z$ by actual rectangles, as desired.

For $z' < z = f(i,j)$ we obtain an $(i,j)$-local
partition of $Y_{z'}$ by restricting the rectangles of the cover $C$ of $Y_z$ to the
rows from $z-z'$ to $z$. This yields an $(i,j)$-local partition of 
a down-shifted copy $Y'$ of $Y_{z'}$.

\bigskip

Now, let us prove \cref{enum:argument}. Due to
\cref{lem:reduction-to-Yk} it is sufficient to show that for
$i,j \in \mathbb{N}$ the Young diagram $Y_{z'}$ with $z' \geq \binom{i+j}{i}$
admits no $(i,j)$-local cover. If $Y_{z'}$ with $z' > z = \binom{i+j}{i}$
has an $(i,j)$-local cover, then by restricting the rectangles of the cover to the
rows from $z'-z$ to $z'$ we obtain an $(i,j)$-local cover of 
a down-shifted copy of $Y_z$. Therefore, we only have to consider
$Y_z$.

Let $C$ be a cover of $Y_z$.
We shall prove that $C$ is not $(i,j)$-local.
Again, we proceed by induction on $i$ and $j$, where illustrations are given in~\cref{fig:argument}.

\begin{figure}
 \centering
 \includegraphics{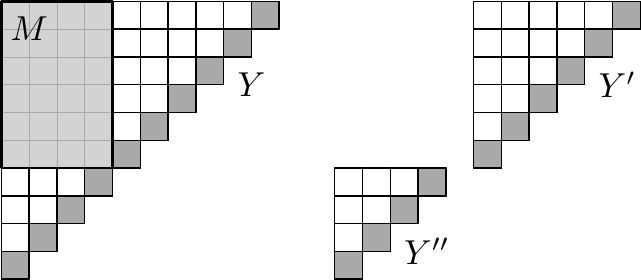}
 \caption{The Young diagram $Y_z$ with $z = \binom{3+2}{3} = 10$ steps, the rectangle $M = [a] \times [z-a] = [6] \times [4]$ with $a = \binom{2+2}{2} = 6$, and the Young diagrams $Y'$ and $Y''$ with $\binom{2+2}{2} = 6$ and $\binom{3+1}{3} = 4$ steps, respectively.}
 \label{fig:argument}
\end{figure}

If $i=1$, then each row is only used by a single rectangle in $C$,
otherwise,~$C$ would not be $(1,j)$-local. Hence, each row of $Y_z$ is a
rectangle in $C$. Thus column~$1$ of $Y_z$ is used by $z = j+1$
rectangles, proving that $C$ is not $(i,j)$-local.

The case $j=1$ is symmetric to the previous by exchanging rows and columns.

Now let $i \geq 2$ and $j \geq 2$.  We have
$z = \binom{i+j}{i} = \binom{(i-1)+j}{i-1} + \binom{i+(j-1)}{i}$.
Consider the rectangle $M = [a] \times [z-a]$ for
$a = \binom{(i-1)+j}{i-1}$.  Then $Y_z - M$
splits into a right-shifted $Y'$ copy of $Y_a$ and a
down-shifted copy $Y''$ of $Y_{z-a}$.
Note that $z-a = \binom{i+(j-1)}{i}$.

Let $C'$, respectively $C''$, be the subset of rectangles in $C$
using at least one of the rows $1,\ldots,a$ in $Y_z$, respectively
at least one of the columns $1,\ldots,z-a$ in $Y_z$. Note that
$C'\cap C'' = \emptyset$ as each generalized rectangle is contained
in $Y_z$.
  
Prune each rectangle in $C'$ to the columns $z-a+1,\ldots,z$ and each
rectangle in~$C''$ to the rows $a+1,\ldots,z$. This yields
covers of $Y'$ and $Y''$.

The Young diagram $Y'$ is a copy of $Y_a$ and
$a=\binom{(i-1)+j}{i-1}$.  Hence, by induction the pruned cover $C'$ is not $(i-1,j)$-local.
If some column $t$ of $Y'$ is used by at least $j+1$
rectangles in $C'$, this column of $Y_z$ is used by at least $j+1$
rectangles in~$C$, proving that $C$ is not $(i,j)$-local, as desired.
So we may assume that some row~$s$ of $Y'$ is used by at least $i$
rectangles in $C'$.
  
Symmetrically, $Y''$ is a copy of $Y_{z-a}$ and $z-a = \binom{i+(j-1)}{i}$.
Hence, the pruned $C''$ is a cover of $Y''$, which by induction is not
$(i,j-1)$-local, and we may assume that some column $t$ of $Y''$ is
used by at least $j$ rectangles in $C''$.  Hence row $s$ in $Y_z$ is
used by at least $i$ rectangles in $C'$ and column $t$ in $Y_z$ is
used by at least $j$ rectangles in $C''$.  As
$C'\cap C'' = \emptyset$ and element $(s,t)$ is contained in
some rectangle of $C$, either row $s$ of $Y_z$ is used by at least
$i+1$ rectangles or column $t$ of $Y_z$ is used by at least $j+1$
rectangles (or both), proving that $C$ is not $(i,j)$-local.
\end{proof}

Finally, \cref{thm:main-simple} follows from \cref{thm:main-general} by setting $i=j=k$.

\section{Local covering numbers}
\label{sec:local-covering}

In~\cite{KMMSSUW18}, Kim \textit{et al.} introduced the concept of covering a Young diagram with generalized rectangles subject to minimizing the maximum number of rectangles in any row or column.
Their motivation was to investigate the relations between \emph{local difference cover numbers} and \emph{local complete bipartite cover numbers}, which are defined as follows\footnote{Deviating from~\cite{KMMSSUW18}, we follow here the terminology and notation of local covering numbers introduced in~\cite{KU16}.}.

A \emph{difference graph} is a bipartite graph in which the vertices of one partite set can be ordered $a_1,\ldots,a_r$ in such a way that $N(a_i) \subseteq N(a_{i-1})$ for $i=2,\ldots,r$, i.e., the neighborhoods of these vertices along this ordering are weakly nesting.
Equivalently, a bipartite graph $H = (V,E)$ with bipartition $V = A \disver B$, $|A| = r, |B| = c$, is a difference graph if $H$ admits a bipartite adjacency matrix $M = (m_{s,t})_{s \in A, t \in B}$ whose support is a Young diagram $Y \subseteq [r] \times [c]$:
\[
 \forall s \in A, t \in B\colon \qquad 
 \{s,t\} \in E \quad \Leftrightarrow \quad (s,t) \in Y \quad \Leftrightarrow \quad m_{s,t} = 1
\]
Then complete bipartite subgraphs $G$ of $H$ correspond precisely to generalized rectangles $R$ in $Y$.
Rows and columns of $M$ correspond to vertices of $H$ in $A$ and $B$, respectively.

Following the notation in~\cite{KU16}, local covering numbers are defined as follows.
For a graph class $\Guestclass$ and a graph $\Host$, an \emph{injective $\Guestclass$-covering} of $\Host$ is a set of graphs $\Guest_1,\ldots,\Guest_t \in \Guestclass$ with $\Host = \Guest_1 \cup \cdots \cup \Guest_t$.
An injective $\Guestclass$-covering of $\Host$ is \emph{$k$-local} if every vertex of $\Host$ is contained in at most $k$ of the graphs $\Guest_1,\ldots,\Guest_t$, and the \emph{local $\Guestclass$-covering number} of $\Host$, denoted by $\cn{\ell}{\Guestclass}{\Host}$, is the smallest $k$ for which a $k$-local injective $\Guestclass$-cover of $\Host$ exists.
 
Let $\Differences$ denote the class of all difference graphs, and $\CBipartites \subset \Differences$ the class of all complete bipartite graphs.
Clearly, we have $\cn{\ell}{\Differences}{\Host} \leq \cn{\ell}{\CBipartites}{\Host}$ for all graphs~$\Host$.
Kim \textit{et al.}~\cite{KMMSSUW18} asked whether there is a sequence of graphs $(\Host_i \colon i \in \mathbb{N})$ for which $\cn{\ell}{\Differences}{\Host_i}$ is constant while $\cn{\ell}{\CBipartites}{\Host_i}$ is unbounded.
They prove that for all graphs $\Host$ on $n$ vertices,
\[
 \cn{\ell}{\CBipartites}{\Host} \leq \cn{\ell}{\Differences}{\Host} \cdot \left\lceil \log_2(n/2+1) \right\rceil,
\]
by showing that
$\cn{\ell}{\CBipartites}{\Host} \leq \lceil \log_2(r+1) \rceil$
whenever $\Host \in \Differences$ is a difference graph with one
partite set of size $r$.  However, no lower bound on
$\cn{\ell}{\CBipartites}{\Host}$ for $\Host \in \Differences$ is
established in~\cite{KMMSSUW18}.  Specifically, Kim \textit{et al.}
ask for the exact value of $\cn{\ell}{\CBipartites}{\Host_i}$ for the
difference graph $H_i$ corresponding to the Young diagram $Y_i$.
For the case that $i+1$ is a power of $2$ they prove the upper bound
$\cn{\ell}{\CBipartites}{\Host_i} \leq \log_2(i+1)-1$.

\medskip

Using \cref{thm:main-simple} and $\binom{2k}{k} = (1 + o(1))\frac{1}{\sqrt{k\pi}}2^{2k}$, we see that
\begin{itemize}
 \item for every difference graph $\Host$ the exact value of $\cn{\ell}{\CBipartites}{\Host}$ is the smallest $k \in \mathbb{N}$ such that for the number $z$ of steps\footnote{In terms of graphs, this is the number of different sizes of neighborhoods in one partite set.} of $\Host$ it holds $z < \binom{2k}{k}$,
 
 \item the difference graphs $\Host_i$, $i \in \mathbb{N}$, defined by Kim \textit{et al.} satisfy
  \[
   \cn{\ell}{\CBipartites}{\Host_i} = (1 + o(1)) \frac{1}{2}\log_2 i,
  \]
  
 \item for this sequence $(\Host_i \colon i \in \mathbb{N})$ of difference graphs $\cn{\ell}{\Differences}{\Host_i}$ is constant $1$, while $\cn{\ell}{\CBipartites}{\Host_i}$ is unbounded, and
  
 \item for all graphs $\Host$ on $n$ vertices,
  \[
   \cn{\ell}{\CBipartites}{\Host} \leq \cn{\ell}{\Differences}{\Host} \cdot (1 + o(1)) \frac{1}{2}\log_2 (n/2).
  \]
\end{itemize}

\section{Local dimension of posets}
\label{sec:local-dimension}

The motivation for Kim \textit{et al.}~\cite{KMMSSUW18} to study local difference cover numbers comes from the local dimension of posets, a notion recently introduced by Ueckerdt~\cite{U16}.

For a partially ordered set (short poset) $\mathcal{P} = (P,\leq)$, define a \emph{partial linear extension} of $\mathcal{P}$ to be a linear extension $L$ of an induced subposet of $\mathcal{P}$.
A \emph{local realizer} of $\mathcal{P}$ is a non-empty set $\mathcal{L}$ of partial linear extensions such that
\textbf{(1)} if $x < y$ in $\mathcal{P}$, then $x < y$ in some $L \in \mathcal{L}$, and 
\textbf{(2)} if $x$ and $y$ are incomparable (denoted $x || y$), then $x < y$ in some $L \in \mathcal{L}$ and $y < x$ in some $L' \in \mathcal{L}$.
The \emph{local dimension} of $\mathcal{P}$, denoted $\ldim(\mathcal{P})$, is then the smallest $k$ for which there exists a local realizer $\mathcal{L}$ of $\mathcal{P}$ with each $x \in P$ appearing in at most $k$ partial linear extensions $L \in \mathcal{L}$.

For an arbitrary height-two poset $\mathcal{P} = (P,\leq)$, Kim \textit{et al.} consider the bipartite graph $G_{\mathcal{P}} = (P,E)$ with partite sets $A = \min(\mathcal{P})$ and $B = P - \min(\mathcal{P}) \subseteq \max(\mathcal{P})$ whose edges correspond to the so-called critical pairs:
\[
 \forall x \in A, y \in B \colon \qquad 
 \{x,y\} \in E \quad \Leftrightarrow \quad x || y \text{ in } \mathcal{P}
\]
They prove that
\[
 \cn{\ell}{\Differences}{G_{\mathcal{P}}} - 2 \leq \ldim(\mathcal{P}) \leq \cn{\ell}{\CBipartites}{G_{\mathcal{P}}} + 2,
\]
which also gives good bounds for $\ldim(\mathcal{P})$ when $\mathcal{P}$ has larger height, since we have
\[
 \ldim(\mathcal{Q})-2 \leq \ldim(\mathcal{P}) \leq 2\ldim(\mathcal{Q})-1
\]
for the associated height-two poset $\mathcal{Q}$ known as the split
of $\mathcal{P}$ (see~\cite{BCPSTT17}, Lemma 5.5). Using these
results and the ones from the previous section, we can conclude the
following for the local dimension of any poset.

\begin{corollary}
 For any poset $\mathcal{P}$ on $n$ elements with split $\mathcal{Q}$ we have
 \[
  \cn{\ell}{\Differences}{G_{\mathcal{Q}}}-4 \leq \ldim(\mathcal{P}) \leq \cn{\ell}{\Differences}{G_{\mathcal{Q}}} \cdot (1 + o(1))\log_2 n.
 \]

\end{corollary}

\section{Ferrers Dimension}
\label{sec:ferrers}
\def\PP{\mathcal{P}}

The aim of this section is to provide some links to
research where related things have been investigated with a
different terminology.

A \emph{Ferrers diagram} is a Young diagram. Typically  Ferrers
diagrams are defined as graphical visualizations of integer
partitions.

Riguet~\cite{Ri51} defined a \emph{Ferrers relation}\footnote{
  According to~\cite{EFO08} Ferrers relations have also been studied
  under the names of biorders, Guttman scales, and bi-quasi-series.
} 
as a relation $R \subset X \times Y$ such that
\begin{quote}
  $(x,y) \in R$ and $(x',y') \in R$ $\quad\implies\quad$ $(x,y') \in R$ or $(x',y) \in R$.
\end{quote}
A relation $R \subset X \times Y$ can be viewed as a digraph $D$ with
$V_D = X\cup Y$ and $E_D=R$. A digraph thus corresponding to a Ferrers
relation is a \emph{Ferrers digraph}. Riguet characterized Ferrers
digraphs as those in which the sets $N^+(v)$ of out-neighbors are linearly
ordered by inclusion. Hence, bipartite Ferrers digraphs are exactly
the difference graphs.

By playing with $x=x'$ and/or $y=y'$ in the definition of a Ferrers
relation it can be shown that Ferrers digraphs without loops are
\textbf{2+2}-free and transitive, i.e., they are interval orders.
In general, however, Ferrers digraphs are allowed to have loops. 

In the spirit of order dimension the \emph{Ferrers dimension of a
  digraph} $D$ ($\textrm{fdim}(D)$) is the minimum number of Ferrers
digraphs whose intersection is~$D$. If $\PP=(P,\leq)$ is poset and
$D_\PP$ the digraph associated with the order relation (reflexivity
implies that $D_\PP$ has loops at all vertices), then
$\dim(\PP) = \textrm{fdim}(D_\PP)$. This was shown by
Bouchet~\cite{Bo71} and Cogis~\cite{Co82}, it implies that Ferrers
dimension is a generalization of order dimension. Since Ferrers
digraphs are characterized by having a staircase shaped adjacency
matrix the complement of a Ferrers digraph is again a Ferrers digraph.
Therefore, instead of representing a digraph as intersection of
Ferrers digraphs containing ($D = \bigcap F_i$ with $D\subseteq F_i$).
We can as well represent its complement as union of Ferrers digraphs
contained in it ($\ovl{D} = \bigcap \ovl{F_i}$ with
$\ovl{F_i}\subseteq \ovl{D}$). This simple observation is sometimes
useful.

The \emph{Ferrers dimension of a relation} $R$ ($\textrm{fdim}(R)$) is
the minimum number of Ferrers relations whose intersection is $R$.
Note that if $D$ is the digraph corresponding to a relation $R$, then
$\textrm{fdim}(D) = \textrm{fdim}(R)$. Hence, the
result of Bouchet can be expressed as $\dim(\PP) =
\textrm{fdim}(P,P,\leq)$, here we use the notation $(P,P,\leq)$ to
emphasize that we interpret the order as a relation.
The interval dimension $\textrm{idim}(\PP)$ of a poset~$\PP$ is the
minimum number of interval orders extending $\PP$ whose intersection
is~$\PP$.  Interestingly interval dimension is
also nicely expressed as a special case of Ferrers dimension:
$\textrm{idim}(\PP) = \textrm{fdim}(P,P,<)$.  For this and far reaching
generalizations see Mitas~\cite{Mi95}.

Relations $R \subset X \times Y$ with $X \cap Y=\emptyset$ can be
viewed as bipartite graphs. In this setting
$\textrm{fdim}(R)$ is the global $\Differences$-covering number of
$\ovl{R}$, i.e., minimum number of difference graphs whose union is
the bipartite complement of $R$.

We believe that it is worthwhile to study local variants of Ferrers
dimension.

\section*{Acknowledgments}
 
This research has been mostly conducted during the Graph Drawing
Symposium 2018 in Barcelona. 
Special thanks go to Peter Stumpf for helpful comments and
discussions.

\bibliography{lit}
\bibliographystyle{plain}

\end{document}